\documentclass[12pt]{article}

\usepackage[utf8]{inputenc}
\usepackage[T1]{fontenc}
\usepackage{amsmath, amssymb, amsthm}
\usepackage{authblk}

\usepackage{mathrsfs}

\title{\bf Chirality and Racemization on Isotopy Classes of Loops: \\
{A Groupoid-Based Structural Theory}}

\author{\Large Takao Inou\'{e}}

\affil{\large Faculty of Informatics, Yamato University, \\ Osaka, Japan\footnote{Email: inoue.takao@yamato-u.ac.jp; \\ Personal Email: takaoapple@gmail.com \\ [I prefer my personal email address for correspondence.]}}

\date{February 25, 2026}

\theoremstyle{definition}
\newtheorem{definition}{Definition}[section]
\newtheorem{assumption}[definition]{Assumption}

\theoremstyle{plain}
\newtheorem{proposition}[definition]{Proposition}
\newtheorem{lemma}[definition]{Lemma}
\newtheorem{theorem}[definition]{Theorem}

\theoremstyle{remark}
\newtheorem{remark}[definition]{Remark}

\newcommand{\Iso}{\mathrm{Iso}}
\newcommand{\Mir}{\mathrm{Mir}}
\newcommand{\Atp}{\mathrm{Atp}}
\newcommand{\Mlt}{\mathrm{Mlt}}
\newcommand{\Inn}{\mathrm{Inn}}

\begin{document}

\maketitle

\begin{abstract}
We develop a theory of chirality and racemization on isotopy classes of finite
loops, formulated intrinsically within the loop isotopy groupoid understood in
the categorical sense.
Motivated by earlier work on quasigroups \cite{InoueQuasiChirality} and by the
classical medical paradigm of mirror-related enantiomers, we restrict admissible
mirror transitions to those generated by intrinsic, unit-preserving symmetries.
Within this framework, racemization is modeled as a two-state dynamics on
isotopy classes, with an effective rate determined by the presence of
mirror-isotopisms.
Our main result shows that this rate vanishes if and only if no loop isotopism
exists between a loop and its opposite, providing a structural criterion for
chirality.
A strengthened variant based on translation-generated symmetries is discussed
in the appendix.
\end{abstract}

\medskip


\tableofcontents

\section{Introduction}
The present work is a continuation of our previous study on chirality and
racemization phenomena on isotopy classes of quasigroups \cite{InoueQuasiChirality}.
The original motivation stems from a classical medical example, namely the
thalidomide enantiomers, where left- and right-handed mirror forms exhibit
drastically different biological effects while being related by a simple
reflection.
In such situations, the fundamental issue is not merely the static distinction
between mirror images, but the existence (or absence) of a \emph{structural
mechanism} by which a system dynamically loses its handedness through
racemization.

Motivated by this observation, our previous work introduced a two-state
dynamical model on isotopy classes of quasigroups, in which chirality was defined
in terms of the non-existence of mirror-isotopisms.
Within that framework, racemization was interpreted as a symmetry-generated
transition between a structure and its mirror image, and a structural
obstruction to such transitions was identified.
However, quasigroups do not carry a distinguished identity element, and general
isotopies need not preserve loop structure.
From both an algebraic and a conceptual standpoint, this represents a genuine
limitation: a racemization theory that allows transitions outside the class of
loops fails to remain closed under the dynamics and departs from the intuition
of an internally consistent system.

The purpose of the present paper is to resolve this structural gap by developing
a chirality and racemization theory that is formulated \emph{entirely within the
world of loops}.
To this end, we replace the use of global group actions by the \emph{loop
isotopy groupoid}, understood strictly in the categorical sense as a small
category in which all morphisms are invertible.
This groupoid-based formulation allows us to encode isotopy equivalence while
retaining a unit-preserving framework.
Within this setting, mirror transitions are no longer introduced externally,
but are restricted to those generated by intrinsic loop symmetries, namely
autotopisms.
As a result, the dynamics remains closed in the class of loops and faithfully
reflects the original physical intuition behind racemization.

The main contributions of this paper can be summarized as follows:
\begin{itemize}
\item We formulate a two-state racemization dynamics on isotopy classes of
loops using the loop isotopy groupoid, avoiding any reliance on global group
actions.
\item We introduce a notion of chirality for loops based on the absence of
mirror-isotopisms within the groupoid framework.
\item By restricting admissible mirror transitions to those generated by
intrinsic symmetries (autotopisms), we obtain a structural vanishing criterion
for the effective racemization rate.
\item We show that the racemization rate on isotopy classes vanishes if and only
if no loop isotopism exists between a loop and its opposite.
\item In an appendix, we present a strengthened variant based on
translation-generated symmetries, yielding a strictly stronger obstruction to
chirality loss.
\end{itemize}

In this way, the present work provides a loop-internal and structurally closed
extension of the quasigroup-based theory, clarifying the precise role played by
unit-preserving symmetries in the algebraic modeling of chirality and
racemization.

Our formulation is based on the loop isotopy groupoid, understood purely
in the categorical sense, rather than on a global group action.

Related chirality phenomena for quasigroups were studied in a parallel framework
in \cite{InoueQuasiChirality}; the present work is logically independent and
focuses instead on a formulation closed within the category of loops.

The role of autotopisms, translations, and inner mappings in loop theory
is classical; see, for example, \cite{NagyStrambach,PhillipsVojtechovsky}.

\section{Loop Isotopy Groupoid and Mirror Loops}

\paragraph{Remark on terminology.}
Throughout this paper, the term \emph{groupoid} is used strictly in the
categorical sense, namely as a small category in which every morphism is
invertible.
No topological, smooth, or higher-categorical structure is assumed.

\begin{definition}[Loop isotopy groupoid]
Let $\mathscr L$ be the small groupoid whose objects are loops
$L=(X,\cdot)$ on a fixed finite set $X$, and whose morphisms
$h=(\alpha,\beta,\gamma):L\to M$ are loop isotopisms satisfying
\[
\alpha(x)\circ\beta(y)=\gamma(x\cdot y).
\]
All morphisms are invertible.
We write $L\sim M$ if such a morphism exists.
\end{definition}

\begin{definition}[Opposite loop]\label{def:opposite_functor}
For a loop $L=(X,\cdot)$, define its opposite loop
$L^\#=(X,\diamond)$ by
\[
x\diamond y := y\cdot x.
\]
\end{definition}

\begin{definition}[Mirror-isotopism]
\[
\Mir(L) := \Iso_{\mathscr L}(L,L^\#).
\]
If $\Mir(L)=\varnothing$, we call $L$ \emph{chiral}.
\end{definition}

\section{Two-State Dynamics and Class Functions}

\begin{definition}[Class function]
A function $f:\mathrm{Ob}(\mathscr L)\to\mathbb R$ is called a class function
if $L\sim M$ implies $f(L)=f(M)$.
Equivalently, $f$ factors through $\pi_0(\mathscr L)$.
\end{definition}

\begin{remark}
Here $\pi_0(\mathscr L)$ denotes the set of equivalence classes of loops
under the relation
\[
L \sim M \quad \Longleftrightarrow \quad
\exists\, h : L \to M \text{ in } \mathscr L,
\]
i.e.\ whenever there exists an isotopism between $L$ and $M$.
\end{remark}

\begin{remark}
The equivalence relation underlying $\pi_0(\mathscr L)$
is ordinary isotopy.
Mirror-isotopisms are not included in this relation,
since the distinction between a loop and its mirror
is essential for the notion of structural chirality
developed below. Mirror-isotopisms are treated separately as morphisms
between $L$ and its mirror $L^\#$.
\end{remark}

\begin{definition}[Two-state generator]
Let $k:\pi_0(\mathscr L)\to\mathbb R$ be a fixed function.
Define
\[
\mathcal L f(L)=k([L])\bigl(f(L^\#)-f(L)\bigr).
\]
\end{definition}

\begin{lemma}
The operator $\mathcal L$ preserves class functions.
\end{lemma}

\begin{proof}
If $L\sim M$, then $[L]=[M]$ and hence $k([L])=k([M])$.
Since mirror-isotopy transports along isotopy,
we have $L^\# \sim M^\#$.
Thus, if $f$ is a class function,
\[
(\mathcal L f)(L)
=
(\mathcal L f)(M).
\]
\end{proof}

\begin{proposition}[Descent]
The generator descends to a well-defined operator
on functions on $\pi_0(\mathscr L)$.
\end{proposition}

\begin{proof}
If $L\sim M$, then $L^\#\sim M^\#$ and $k([L])=k([M])$.
Hence the generator agrees on representatives of the same class.
\end{proof}

\begin{proposition}[Descent]
The generator descends to a well-defined operator
on functions on $\pi_0(\mathscr L)$.
\end{proposition}

\begin{remark}
Equivalently, the two-state dynamics depends only on the isotopy class $[L]$
and not on the choice of a representative loop.
\end{remark}

\section{Structural Chirality via Autotopisms}

\begin{remark}
The use of the loop isotopy groupoid, rather than a global group action,
allows mirror-isotopisms to be transported functorially between isotopic loops
without leaving the class of loops.
This categorical viewpoint is essential for the structural arguments below.
\end{remark}

\begin{definition}[Isotopisms]
For loops $L,M$, we write
\[
\Iso(L,M):=\Iso_{\mathscr L}(L,M).
\]
The elements of $\Iso(L,M)$ are called \emph{isotopisms} from $L$ to $M$.
\end{definition}

\begin{definition}[Autotopism group]
\[
\Atp(L)=\Iso_{\mathscr L}(L,L).
\]
The elements of $\Atp(L)$ are called \emph{autotopisms} of $L$.

We write
\[
\Mir(L):=\Iso_{\mathscr L}(L,L^\#),
\]
and call its elements \emph{mirror-isotopisms} of $L$.
\end{definition}

\begin{lemma}[Transport of mirror-isotopisms]\label{lem:transport_mirror}
Let $h:L\to M$ be an isotopism of loops.
For any mirror-isotopism $g\in\Mir(L)=\Iso(L,L^\#)$,
define
\[
\Phi_h(g):=h^\#\circ g\circ h^{-1}.
\]
Then $\Phi_h(g)$ is a mirror-isotopism from $M$ to $M^\#$, and
\[
\Phi_h:\Mir(L)\;\longrightarrow\;\Mir(M)
\]
is a bijection.
In particular, $\Mir(L)=\varnothing$ if and only if $\Mir(M)=\varnothing$.
\end{lemma}

\begin{proof}
By Definition~\ref{def:opposite_functor} (Section~2),
the mirror operation $(\cdot)^\#$ defines an involutive functor
on the loop isotopy groupoid $\mathscr L$.
In particular, if $h:L\to M$ is an isotopism, then
\[
h^\#:L^\#\to M^\#
\]
is also an isotopism.

Let $g\in\Mir(L)=\Iso(L,L^\#)$ be a mirror-isotopism of $L$.
Since $h$ is invertible in the groupoid $\mathscr L$,
its inverse $h^{-1}:M\to L$ is an isotopism.
Therefore the composite
\[
h^\#\circ g\circ h^{-1}:M\longrightarrow M^\#
\]
is well-defined and is an element of $\Iso(M,M^\#)=\Mir(M)$.
This defines a map
\[
\Phi_h:\Mir(L)\longrightarrow\Mir(M),
\qquad
\Phi_h(g)=h^\#\circ g\circ h^{-1}.
\]

Since the mirror functor is involutive, we have
\[
(h^{-1})^\#=(h^\#)^{-1}.
\]
It follows that the inverse map of $\Phi_h$ is given by
$\Phi_{h^{-1}}$, and hence $\Phi_h$ is a bijection.
\end{proof}
\begin{assumption}[Symmetry-generated transitions]
Mirror transitions are generated solely by elements of $\Mir(L)$,
with strictly positive weights
\[
w_L:\Mir(L)\to(0,\infty),
\]
satisfying transport invariance under isotopy.
\end{assumption}

\begin{remark}\label{rem:positive_weights}
The assumption that all weights are strictly positive is essential.
If merely nonnegative weights were allowed, the vanishing of the total rate
could occur by cancellation, even in the presence of mirror-isotopisms.
The present assumption excludes such accidental degeneracies and ensures that
$k([L])=0$ reflects a genuine structural obstruction.
\end{remark}

\begin{proposition}[Class invariance]
\[
k([L])=\sum_{g\in\Mir(L)} w_L(g)
\]
depends only on the isotopy class $[L]$.
\end{proposition}

\begin{proof}
If $h:L\to M$ is an isotopism,
Lemma~\ref{lem:transport_mirror} gives a bijection
$\Mir(L)\to\Mir(M)$.
By transport invariance of the weights,
corresponding terms have equal weight.
Hence the sums defining $k([L])$ and $k([M])$ coincide.
\end{proof}

\begin{theorem}[Structural vanishing criterion]
The following are equivalent:
\begin{enumerate}
\item $k([L])=0$,
\item $\Mir(L)=\varnothing$,
\item there is no loop isotopism $L\to L^\#$.
\end{enumerate}
\end{theorem}

\begin{proof}
Recall that, by definition, $k([L])$ is given as the total weight
\[
k([L])=\sum_{g\in\Mir(L)} w_L(g),
\]
where the index set is $\Mir(L)=\Iso_{\mathscr L}(L,L^\#)$.
Assume that $\Mir(L)$ is finite and that each weight satisfies $w_L(g)>0$.

If $\Mir(L)=\varnothing$, then the sum is an empty sum and hence equals $0$.
Therefore $k([L])=0$.

Conversely, suppose that $\Mir(L)\neq\varnothing$.
Since $\Mir(L)$ is finite, we may write $\Mir(L)=\{g_1,\dots,g_n\}$ with $n\ge 1$.
Then
\[
k([L])=\sum_{i=1}^n w_L(g_i),
\]
and each term is strictly positive, so the whole sum is strictly positive.
Hence $k([L])>0$, and in particular $k([L])\neq 0$.
This proves that $k([L])=0$ holds if and only if $\Mir(L)=\varnothing$,
i.e.\ (1) and (2) are equivalent.

Finally, (2) and (3) are equivalent because
\[
\Mir(L)=\Iso_{\mathscr L}(L,L^\#)
\]
by definition: $\Mir(L)$ is empty precisely when there exists no isotopism
from $L$ to its mirror $L^\#$ in the groupoid $\mathscr L$.
\end{proof}

\section{Examples}

\paragraph{Groups.}
Let $L$ be a group.
Then the inversion map $x\mapsto x^{-1}$ induces an isotopism between
$L$ and its opposite loop $L^\#$.
Consequently, $\Mir(L)\neq\varnothing$, and the effective racemization rate
$k([L])$ is strictly positive.
In particular, groups are always achiral in the sense of the present theory.
This is consistent with the fact that group structure admits abundant
unit-preserving symmetries.

\paragraph{Non-associative loops.}
For non-associative loops, the situation is markedly different.
In general, there need not exist any loop isotopism between a loop $L$
and its opposite $L^\#$.
When $\Mir(L)=\varnothing$, Theorem~4.7 implies that the effective
racemization rate vanishes identically, and $L$ is structurally chiral.
Such examples arise naturally among non-associative loops with restricted
autotopism groups, where intrinsic symmetries are too sparse to generate
mirror transitions.

\paragraph{Comparison with the quasigroup case.}
In the quasigroup setting, isotopies need not preserve the existence of a
distinguished identity element, and mirror transitions may occur through
isotopies that leave the class of loops.
The present loop-internal formulation excludes such external mechanisms.
As a result, structural chirality is governed purely by unit-preserving
symmetries, clarifying the precise role played by autotopisms in obstructing
or enabling racemization.

\section{Conclusion}

We have established a loop-internal chirality theory based on the
loop isotopy groupoid.
The effective racemization rate vanishes precisely when no mirror-isotopism exists.

In closing, we briefly return to the original medical motivation that inspired
this line of research.
In the case of thalidomide, the coexistence of mirror-related enantiomers with
radically different biological effects highlights the importance of understanding
not only static chirality, but also the mechanisms by which chirality may be lost
or preserved under dynamical processes.
From an abstract algebraic viewpoint, this corresponds to asking whether a
structure admits an internal symmetry mechanism that identifies it with its
mirror image.
The loop-theoretic framework developed in this paper provides a precise
mathematical analogue of this question: structural chirality is preserved exactly
when no unit-preserving isotopy exists between a loop and its opposite.
In this sense, the present theory offers a conceptually faithful abstraction of
racemization phenomena, demonstrating how the presence or absence of intrinsic
symmetries governs the stability of handedness at a purely structural level.

\appendix

\section{Strengthening: Translation-Generated Variant}

\subsection{Translations and inner mappings}

For $L=(X,\cdot)$ define left and right translations
\[
L_x(y)=x\cdot y, \quad R_x(y)=y\cdot x.
\]
Let $\Mlt(L)$ be the multiplication group generated by all
$L_x$ and $R_x$, and $\Inn(L)$ its inner mapping subgroup.

\subsection{Stronger admissibility}

One may restrict admissible mirror-isotopisms to those generated
by translations or inner mappings.
This produces a strictly stronger obstruction:
if no translation-generated mirror-isotopism exists,
then $L$ is structurally chiral.

$$ $$

\noindent Takao Inou\'{e}

\noindent Faculty of Informatics

\noindent Yamato University

\noindent Katayama-cho 2-5-1, Suita, Osaka, 564-0082, Japan

\noindent inoue.takao@yamato-u.ac.jp
 
\noindent (Personal) takaoapple@gmail.com (I prefer my personal mail)

\end{document}